\documentclass[a4paper,11pt]{scrartcl}

\usepackage[british]{babel}
\usepackage{amssymb}
\usepackage{amsthm}
\usepackage{amsmath}
\usepackage{enumerate}
\usepackage[utf8]{inputenc}
\usepackage[T1]{fontenc}

\newtheoremstyle{thm}{\topsep}{\topsep}%
     {\slshape}
     {}
     {\bfseries}
     {:}
     { }
     {\thmname{#1\,}\thmnumber{#2\,}\thmnote{(#3)}}
\newtheoremstyle{def}{\topsep}{\topsep}%
	 {}
	 {}
	 {\bfseries}
	 {:}
	 { }
	 {\thmname{#1\,}\thmnumber{#2\,}\thmnote{(#3)}}

\theoremstyle{thm}
\newtheorem{theorem}{Theorem}[section]
\newtheorem{lemma}[theorem]{Lemma}

\newtheorem{corollary}[theorem]{Corollary}
\theoremstyle{def}

\newtheorem{remark}[theorem]{Remark}
\newtheorem{example}[theorem]{Example}

\newcommand{\tens}[1]{\boldsymbol{\mathsf{#1}}}
\newcommand{\vect}[1]{\boldsymbol{#1}}
\newcommand{\Ll}{\ensuremath{L_{\mathrm{\scriptstyle loc}}}}
\newcommand{\Hl}{\ensuremath{H_{\mathrm{\scriptstyle loc}}}}
\newcommand{\Wl}{\ensuremath{W_{\mathrm{\scriptstyle loc}}}}
\newcommand{\GT}{\Hl^1\cap\Ll^\infty}

\DeclareMathOperator{\Div}{div}
\DeclareMathOperator{\Grad}{grad}
\DeclareMathOperator{\Tr}{tr}

\begin{document}

\title{Wave equations and symmetric first-order systems in case of low regularity}
\author{Clemens Hanel\thanks{E-Mail: clemens.hanel@univie.ac.at} \\ Günther Hörmann\thanks{E-Mail: guenther.hoermann@univie.ac.at} \\ Christian Spreitzer\thanks{E-Mail: christian.spreitzer@univie.ac.at} \\ Roland Steinbauer\footnote{E-Mail: roland.steinbauer@univie.ac.at} \\ \small{University of Vienna, Faculty of Mathematics} \\[-0.8ex] \small{Nordbergstraße 15, 1090 Wien, Austria}}

\date{$2^\text{nd}$ of February, 2012}

\maketitle

\begin{abstract} 
\noindent We analyse an algorithm of transition between Cauchy problems for second-order wave equations and first-order symmetric hyperbolic systems in case the coefficients as well as the data are non-smooth, even allowing for regularity below the standard conditions guaranteeing well-posedness. The typical operations involved in rewriting equations into systems are then neither defined classically nor consistently extendible to the distribution theoretic setting. However, employing the nonlinear theory of generalized functions in the sense of Colombeau we arrive at clear statements about the transfer of questions concerning solvability and uniqueness from wave equations to symmetric hyperbolic systems and vice versa. Finally, we illustrate how this transfer method allows to draw new conclusions on unique solvability of the Cauchy problem for wave equations with non-smooth coefficients.\\[1\baselineskip]
\end{abstract}

	\section{Introduction}

	Theories for higher order partial differential equations on the one hand and first-order systems of (pseudo)differential equations on the other hand are to a large extent developed in parallel, although elaborate mechanisms for rewriting the former into the latter do exist in terms of modern analysis (cf.\ \cite{Kum81,Tay81}). However, in general the transition methods require high-powered pseudodifferential operator techniques and, what is even more restrictive in special situations,  do often require a certain smoothness of the coefficients (or symbols) to be mathematically meaningful in all their intermediate operations beyond mere formal manipulation. Nonlinear theories of generalized functions, in particular the differential algebras constructed in the sense of Colombeau, provide a means to embed distributions into a wider context where the transition between higher order equations and first-order systems can be based on well-defined operations. Thus, Colombeau theory allows to rigorously address the question about the precise relation between generalized solutions to wave equations and those of corresponding hyperbolic first-order systems in case of non-smooth coefficients. 

	The theory of generalized solutions to linear hyperbolic first-order systems has been developed over 20 years and has achieved a spectrum of results on existence and uniqueness of solutions to the Cauchy problem, distributional limits and regularity of solutions, and symmetrizability  (cf.\ \cite{Obe89,Obe92,LaOb91,Obe09,HoSp12,GaOb11b}). 

        On the other hand, generalized solutions of wave equations arising via the Laplace-Beltrami operator of a Lorentzian metric of low regularity have been studied in \cite{ViWi00,May06,GMS09,Han11,HKS11}. These investigations draw strong motivation from general relativity, in particular in the context of Chris Clarke's notion of generalized hyperbolicity \cite{Cla96,Cla98}, which generalizes the classical notion of global hyperbolicity (i.\,e. the geometric condition necessary for global well-posedness of the Cauchy problem for wave equations). More precisely, local and global existence and uniqueness of generalized solutions have been established for a wide class of ``weakly singular'' space-time metrics which are described using the geometric theory of nonlinear generalized functions (\cite[Chapter 3]{GKOS01}).
        
        In this paper we establish a rigorous way to translate Cauchy problems for wave equations into such for symmetric hyperbolic systems and vice versa in case of low regularity, thereby making results from either theory potentially available to the other. Also we show some of this potential by inspecting results on wave equations obtained from statements on first-order systems through careful analysis of the translation process.

	The plan of the paper is as follows. Section \ref{Notions} introduces and briefly reviews the required basic notions from Colombeau's theory of generalized functions and numbers. In Section \ref{Transformation} we present an explicit method to transform a second-order wave equation with generalized function coefficients into a symmetric hyperbolic system of first order and describe in precise terms the relation between generalized solutions in either case. The main results here are summarized in Theorem \ref{equi_thm} and in the simple Example \ref{Gegenbeispiel} we illustrate what kind of difficulties from the pure distribution theoretic viewpoint are remedied by our result. As an application we show in Section \ref{Existence} that solvability results on symmetric hyperbolic systems can be used to deduce in Theorem \ref{Th:Wave from system} new aspects on solvability of the Cauchy problem for wave equations with non-smooth coefficients.

	\section{Basic notions and spaces}\label{Notions}
	\paragraph{Notation}
	We assume several notational conventions to keep calculations clearly laid out: We denote vector valued functions by bold symbols, e.\,g. $\vect v$, and matrix valued functions by bold and sans serif letters, e.\,g. $\tens R$. We write $v_i$ for the components of a vector $\vect v$ and $R_{ij}$ for the components of a matrix $\tens R$. The $i^{\text{th}}$ row respectively $j^{\text{th}}$ column of a matrix $\tens R$ is denoted by $\vect R_{i\cdot}$ or $\vect R_{\cdot j}$ respectively. The spatial gradient of a scalar function $u$ shall be $\vect u'=\Grad u=\partial_x u$, the spatial Hessian shall be $\boldsymbol{\mathsf u}''$. We denote the Euclidean scalar product by $\langle\cdot,\cdot\rangle$.
	\medskip 

	\paragraph{Generalized functions}
	We will use variants of Colombeau algebras as presented in \cite{Col84,Obe92,GKOS01,Gar05}. Here, we recall their 
        essential features: 
		Let $E$ be a locally convex topological vector space with a topology given by a family of semi-norms $\{p_j\}$ with $j$ in some index set $J$. We define
		\begin{align*}
		 	\mathcal M_E:= & \{(u_\varepsilon)_\varepsilon\in E^{(0,1]}\mid\forall j\in J\:\: \exists N\in\mathbb N_0: p_j(u)=O(\varepsilon^{-N})\text{ as }\varepsilon\to 0\}, \\
			\mathcal N_E:= & \{(u_\varepsilon)_\varepsilon\in E^{(0,1]}\mid \forall j\in J\:\: \forall m\in\mathbb N_0: p_j(u)=O(\varepsilon^{m})\text{ as }\varepsilon\to 0\},
		\end{align*}
		the moderate respectively negligible subsets of $E^{(0,1]}$. Operations are induced from $E$ by $\varepsilon$-wise application, so we have the (vector space) inclusion relation \linebreak $\mathcal N_E\subseteq\mathcal M_E\subseteq E^{(0,1]}$. The generalized functions based on $E$ are defined as the quotient space $\mathcal G_E:=\mathcal M_E/\mathcal N_E$. If $E$ is a differential algebra, then $\mathcal N_E$ is an ideal in $\mathcal M_E$ and therefore $\mathcal G_E$ is a differential algebra as well, called the Colombeau algebra based on $E$.
	
		Let now $U$ be an open subset of $\mathbb R^n$. If we choose $E=\mathcal C^\infty(U)$ with the topology of uniform convergence of all derivatives on compact sets, then we obtain the special Colombeau algebra on $U$, i.\,e. $\mathcal G_{\mathcal C^\infty(U)}=\mathcal G(U)$.
		
		Moreover, we will also use the following three Sobolev spaces in this construction:
		\begin{itemize}
			\item $E=H^\infty(U)=\{u\in\mathcal C^\infty(\overline U)\mid\partial^\alpha u\in L^2(U)\,\forall\alpha\in\mathbb N^n_0\}$ with the family of norms
			\begin{equation*}
				\|u\|_{H^k}=\bigl(\sum_{|\alpha|\leq k}\|\partial^\alpha u\|_{L^2}\bigr)^{\frac 1 2}\quad k\in\mathbb N_0,
			\end{equation*}
			\item $E=W^{\infty,\infty}(U)=\{u\in\mathcal C^\infty(\overline U)\mid\partial^\alpha u\in L^\infty(U)\,\forall\alpha\in\mathbb N^n_0\}$ with the family of norms
			\begin{equation*}
				\|u\|_{W^{k,\infty}}=\max_{|\alpha|\leq k}\|\partial^\alpha u\|_{L^\infty}\quad k\in\mathbb N_0,
			\end{equation*}
			\item $E=\mathcal C^\infty(\overline I\times\mathbb R^n)$, where $I$ is an open, bounded interval, equipped with the family of semi-norms
			\begin{equation*}\
				\|u\|_{k,K}=\max_{|\alpha|\leq k}\|\partial^\alpha u\|_{L^\infty(\overline I\times K)},
			\end{equation*}
			where $K$ is a compact subset of $\mathbb R^n$ and $k\in\mathbb N_0$.
		\end{itemize}
		To simplify notation, we denote the corresponding Colombeau algebras as in \cite{Hor11}:
			\begin{align*}
				\mathcal G_{L^2}(U):= & G_{H^\infty(U)} & \mathcal G_{L^\infty}(U):= & G_{W^{\infty,\infty}(U)} & \mathcal G(\overline I\times\mathbb R^n):= & \mathcal G_{\mathcal C^\infty(\overline I\times\mathbb R^n)}.
			\end{align*}
	
	Elements in $\mathcal G_E({\mathbb R}^{n})$ are denoted by $u=[(u_\varepsilon)_\varepsilon]=(u_\varepsilon)_\varepsilon+\mathcal N_E({\mathbb R}^{n})$. We restrict a Colombeau function $u(t,x)$ to an initial surface by taking $u|_{t=0}=[(u_\varepsilon(0,x))_\varepsilon]$.
			The ring of generalized numbers $\widetilde{\mathbb R}$ consists of elements $u=[(u_\varepsilon)_\varepsilon]$, where $u_\varepsilon\in\mathbb R$. Note that $\widetilde{\mathbb R}^n$ is a module over $\widetilde{\mathbb R}$, a fact we have to keep in mind when doing linear algebra.

Generalized functions in $\mathcal G(\mathbb R^n)$ are characterized by their generalized point values. In fact, considering only classical point values is insufficient as the following statement from \cite{KuOb99} shows, cf. \cite[Thm. 1.2.46]{GKOS01}. Let $f\in\mathcal G({\mathbb R}^n)$. The following are equivalent:
			\begin{enumerate}[(i)]
				\item $f=0$ in $\mathcal G({\mathbb R}^n)$,
				\item $f(\tilde x)=0$ in $\widetilde{\mathbb R}$ for each $\tilde x\in\widetilde{\mathbb R}^n_c$.
			\end{enumerate}
		Here $\widetilde{\mathbb R}^n_c$ denotes the set of compactly supported generalized points: A generalized point $x\in\widetilde{\mathbb R}^n$ is compactly supported if there exists $K\subseteq\widetilde{\mathbb R}^n$ compact and $\eta>0$ such that $x_\varepsilon\in K$ for $\varepsilon<\eta$. 
		
			A matrix valued generalized function $\tens G\in M_k(\mathcal G({\mathbb R}^n))$ is called symmetric and nondegenerate if for any $\tilde x\in\widetilde{\mathbb R}_c^n$ the bilinear map $\tens G(\tilde x):\widetilde{\mathbb R}^k\times\widetilde{\mathbb R}^k\to\widetilde{\mathbb R}$ is symmetric and nondegenerate, \cite[Def. 5.1.2]{GKOS01}
		Here, by nondegenerate we mean that $\vect\xi\in\widetilde{\mathbb R}^n$, $\tens G(\tilde x)(\vect\xi,\vect\eta)=0\,\forall\vect\eta\in\widetilde{\mathbb R}^n$ implies $\vect\xi=0$.
		Apart from this pointwise definition, there exist equivalent characterizations of nondegeneracy, see \cite[Theorem 3.2.74]{GKOS01}. In particular, there always exists a representative entirely consisting of symmetric, nondegenerate matrices. If, in addition, there exists a representative of constant index $j$, we call $j=\nu(\tens G)$ the index of $\tens G$. We call matrices in $M_n(\widetilde{\mathbb R}^n)$ with $j=0$ Riemannian metrics and such with $j=1$ Lorentzian metrics. For concepts of linear algebra in $\widetilde{\mathbb R}^n$ we refer to \cite{May08}.
Finally, we point out the following lemma on Lorentzian metrics (the proof of which is straightforward).\medskip

	\begin{lemma}\label{lorentz_lemma}
		Let $\tens G\in M_{n+1}(\widetilde{\mathbb R}^{n+1})$ be of the form
		\begin{equation*}
			\tens G=\begin{pmatrix}
				-1 & \vect g^T \\
				\vect g & \tens R
			\end{pmatrix},
		\end{equation*}
		 with $\tens R$ a generalized Riemannian metric on $\widetilde{\mathbb R}^n$, 
                 then $\tens G$ is Lorentzian.
	\end{lemma}
	
\section{Transformation between equations and systems}\label{Transformation}

	In this section we relate solutions of wave equations to solutions of corresponding symmetric first-order systems.
	To this end, consider a wave equation in $\mathcal G({\mathbb R}^{n+1})$
		\begin{equation}\label{wave_eq}
		-\partial_t^2 u+2\sum_{i=1}^n g_i\partial_{x_i}\partial_t u+\sum_{i,j=1}^n R_{ij}\partial_{x_i}\partial_{x_j} u + a\partial_t u+\sum_{i=1}^n b_i\partial_{x_i} u +cu =f,
		\end{equation}
		with principal part derived from $\tens G:=\left(\begin{smallmatrix}-1 & \vect g^T \\ \vect g & \tens R\end{smallmatrix}\right)$, a generalized Lorentzian metric. In fact, our arguments still hold in the more general case, where the matrix entry $G_{00}$ is a strictly negative generalized function (i.\,e. $-G_{00}>\varepsilon^m$ on compact sets for some $m>0$), upon dividing by $-G_{00}$. Here $\tens R=(R_{ij})$ is a positive definite, symmetric matrix of generalized functions, $\vect g$ and $\vect b$ are vectors with entries in $\mathcal G({\mathbb R}^{n+1})$ and $a$, $c$, and $f$ are generalized functions.
		 
        Next we rewrite the wave equation \eqref{wave_eq} as a first-order system. There exist several algorithms to obtain a hyperbolic first-order system. However, we employ an algorithm that also guarantees the symmetry of the system. Indeed by setting $\boldsymbol{w}=(u,\partial_t u, \tens S \vect u')^T$ we arrive at the system 
	\begin{equation}\label{system}
		-\partial_t\vect w + \sum_{i=1}^n \boldsymbol{\mathsf A}_i\partial_{x_i}\vect w+\boldsymbol{\mathsf B}\vect w=\boldsymbol F
	\end{equation}
	in $\mathcal G({\mathbb R}^{n+1})^{n+2}$. Here  $\tens S=\tens R^{\frac 1 2}$ is 
	constructed via $\varepsilon$-wise diagonalization. The so constructed square root is a symmetric positive definite matrix with entries in $\mathcal G(\mathbb R^{n+1})$, as we will discuss below. The matrices $\tens A_i$, $\tens B$, and the vector $\vect F$ are given in the following way:
	
	\begin{align}\label{coefficents of system}
	\boldsymbol{\mathsf A}_i= & \left(\begin{array}{lll} 
	0 & 0 & 0_{1\times n} \\
	0 & 2g_i & \vect S_{i\cdot}\\
	0_{n\times 1} & \vect S_{\cdot i} & 0_{n\times n}
	\end{array}\right) & \boldsymbol F = & 
	\left(\begin{array}{c} 0 \\ f \\ 0_{n\times 1}
	\end{array}\right)\notag \\
	\boldsymbol{\mathsf B}= & \left(\begin{array}{llc} 
	0 & 1 & 0_{1\times n}\\
	c & a & (\Div \tens S)^T+(\boldsymbol b-\Div \boldsymbol{\mathsf S}^2)^T\tens S^{-1} \\
	0_{n\times 1} & 0_{n\times 1} & -(\partial_t \tens S)\tens S^{-1} 
	\end{array}\right).
	\end{align}

	Here we have used the fact that $\Tr (\tens S^2 {\boldsymbol{\mathsf u}}'')= \Div (\tens S^2 \vect u')-\langle\Div \tens S^2,\vect u'\rangle$. A word on the notation is in order. For any symmetric matrix $\tens S$, we set the divergence\linebreak
	$(\Div \tens S)_i=\sum_{j=1}^n \partial_{x_j}S_{ij}$, i.\,e. $\Div \tens S$ is a vector, whose $i^\text{th}$ entry is simply the divergence of the $i^\text{th}$ row (or column) of the matrix $\tens S$.
	
	Finally, we show that $\tens S=\tens R^{\frac 1 2}$ is a symmetric and positive definite matrix with generalized functions as entries. We can write a representative of this matrix $\tens{S}$ via diagonalization, so we have $\tens S_\varepsilon(t,x)=\tens U_\varepsilon(t,x)^T{\tens D_\varepsilon(t,x)}^{\frac 1 2}\tens U_\varepsilon(t,x)$ with $$\tens D_\varepsilon(t,x)^{\frac 1 2}=\mathrm{diag} \Bigl(\sqrt{\lambda_{1,\varepsilon}(t,x)},\dots,\sqrt{\lambda_{n,\varepsilon}(t,x)}\Bigr),$$ where $\lambda_{1,\varepsilon}(t,x),\dots,\lambda_{n,\varepsilon}$ are the eigenvalues of a symmetric and positive definite representative $(\tens{R}_\varepsilon)_\varepsilon$ of $\tens{R}$. Observe that $(\tens D_\varepsilon)_\varepsilon$ and $(\tens U_\varepsilon)_\varepsilon$ are not necessarily nets of matrices with smooth entries, however, the product $(\tens U_\varepsilon^T{\tens D}^{\frac 1 2}_\varepsilon\tens U_\varepsilon)_\varepsilon$ is smooth by the following lemma (where we denote by $S_n(\mathbb R)$ and $S^+_n(\mathbb R)$ the spaces of symmetric and positive definite symmetric matrices in $M_n(\mathbb R)$).
	\medskip

\begin{lemma}\label{Smooth square root}
	The smooth map $f:S^+_n(\mathbb R)\to S^+_n(\mathbb R)$ with $f(\tens A)=\tens A^2$ is a diffeomorphism.
\end{lemma}
\begin{proof}
	The map $f:S^+_n(\mathbb R)\to S^+_n(\mathbb R)$ is bijective (e.\,g.\ \cite[Prop.\ 6.8]{Lan02}) and we employ the inverse function theorem to conclude that $f$ is a global diffeomorphism. Indeed since $S^+_n(\mathbb R)$ is an open subset of $S_n(\mathbb R)$ we may identify the tangent space $\mathrm T_{\tens A}S^+_n(\mathbb R)$ for $\tens A\in S^+_n(\mathbb R)$ with $S_n(\mathbb R)$ and obtain $\mathrm df(\tens A)(\tens B)=\tens A\tens B+\tens B\tens A$. Now injectivity of $\mathrm df(\tens A)$ follows since if $\tens A\tens B=-\tens B\tens A$ and $\tens B\not=0$ there is $0\not=\lambda$ with $\tens{B}\vect v=\lambda \vect v$ for some $\vect v\not=0$. But then $\tens A\vect v$ is an eigenvector to $-\lambda$ and by symmetry of $\tens{B}$ we have $\langle\vect v,A\vect v\rangle=0$, contradicting positive definiteness of $\tens{A}$.
\end{proof}

	\paragraph{From the wave equation to the first-order system}
	
 Assuming that we have a solution $u$ to the wave equation \eqref{wave_eq}, in the following lemma we will construct a solution to the first-order system \eqref{system}. Basically, we define a vector $\vect w=(u,\partial_t u, \tens S \vect u')^T$ and rewrite the wave equation in terms of the three components of $\vect w$.\medskip

\begin{lemma}\label{wave to system}
Let $u_0,u_1\in \mathcal G(\mathbb R^n)$ and consider the second-order equation \eqref{wave_eq} with initial condition $(u,\partial_t u)|_{t=0}=(u_0,u_1)$. If $u\in\mathcal G(\mathbb R^{n+1})$ is a solution of \eqref{wave_eq}, then
the vector $\boldsymbol{w}=(u,\partial_t u, \tens S \vect u')^T$ is a solution to the first-order system \eqref{system} with initial condition $\vect w|_{t=0}=(u_0,u_1,\tens S \vect u_0')^T$. 
\end{lemma}

\begin{proof}

First we rewrite equation \eqref{wave_eq} in divergence form, i.\,e.
$$ -\partial_t^2 u+2\langle \boldsymbol{g},\partial_t \vect u'\rangle+\Div (\tens S^2 \vect u') + a\partial_t u+\langle \boldsymbol b-\Div \tens S^2,\vect u'\rangle +cu=f.
$$

Introducing new variables $z:=\partial_t u$ and $\vect v:= \tens S \vect u'$, we may write
$$ -\partial_t z+2\langle \boldsymbol{g},\vect z'\rangle+\Div (\tens S \vect v) + az+\langle \boldsymbol b-\Div \tens S^2,\tens S^{-1}\vect v\rangle +cu=f.
$$
Hence we have 
\begin{align*} -\partial_t\vect w + & \sum_{i=1}^n
\left(\begin{array}{lll} 
0 & 0 & 0_{1\times n} \\
0 & 2g_i & \vect S_{i\cdot}\\
0_{n\times 1} & \vect S_{\cdot i} & 0_{n\times n}
\end{array}\right)\partial_{x_i}\vect w\notag \\
+ & \left(\begin{array}{llc} 
0 & 1 & 0_{1\times n}\\
c & a & (\Div \tens S)^T+(\boldsymbol b-\Div \tens S^2)^T\tens S^{-1} \\
0_{n\times 1} & 0_{n\times 1} & (\partial_t \tens S)\tens S^{-1} 
\end{array}{}\right)\vect w=
\left(\begin{array}{c} 0 \\ f \\ 0_{n\times 1}
\end{array}\right),
\end{align*}
where we have used that $\Div (\tens S\vect v)=\sum_{i=1}^n \vect S_{i\cdot}\partial_{x_i}\vect v +(\Div \tens S)^T\vect v$. Note that the second equation in the above system is just the original wave equation written in new variables, 
whereas the other equations represent the transformation of variables.
Finally, evaluating $\vect w$ at time $t=0$ yields $\vect w|_{t=0}=(u_0,u_1,\tens S \vect u')^T$.
\end{proof}

\paragraph{From the first-order system to the wave equation}

We now look at the converse situation: Given a solution to the first-order system \eqref{system}, we would like to prove existence for wave-type equations. Now, let $\tens S$ be a symmetric and invertible $n$-dimensional matrix with entries in $\mathcal G(\mathbb R^{n+1})$, let $\vect g$ and $\vect b$ be vectors with entries in $\mathcal G({\mathbb R}^{n+1})$, and let $a,c$ be generalized functions. Observe that the matrix $\boldsymbol{\mathsf B}$ is not restricted by the structure of the term $(\Div \tens S)^T+(\boldsymbol b-\Div \tens S^2)^T\tens S^{-1}$. Let \linebreak $\tilde{\vect b}^T=(\Div \tens S)^T+(\boldsymbol b-\Div \tens S^2)^T\tens S^{-1}$. Multiplication with $\tens S$ from the right gives
\begin{equation*}
	\tilde{\vect b}^T\tens S=(\Div \tens S)^T\tens S+(\boldsymbol b-\Div \tens S^2)^T.
\end{equation*}
Bringing all terms except the one containing $\vect b$ to the other side results in
\begin{equation*}
	\tilde{\vect b}^T\tens S-(\Div \tens S)^T\tens S+(\Div \tens S^2)^T=\vect b^T.
\end{equation*}
Finally, transposition leads to an equation for $\vect b$ entirely in terms of $\tens S$ and an arbitrarily chosen coefficient $\tilde{\vect b}$:
\begin{equation*}
	\vect b=\tens S^T\tilde{\vect b}-\tens S^T\Div \tens S+(\Div \tens S^2).
\end{equation*}

\begin{lemma}\label{system to wave}

Let $u_0,u_1\in \mathcal G(\mathbb R^n)$. If $\vect w=(u,z,\vect v)^T \in \mathcal G(\mathbb R^{n+1})^{n+2}$ is a solution to the first-order system \eqref{system}
with initial condition $\vect w|_{t=0}=(u_0,u_1,\tens S \vect u_0')^T$,
then $u$ is a solution to \eqref{wave_eq} with initial condition $(u,\partial_t u)|_{t=0}=(u_0,u_1)$.

\end{lemma}

\begin{proof} 

Note that the first equation of our system is just $z=\partial_t u$. The last $n$ equations read
\[
	-\partial_t \vect v+\tens S z' +(\partial_t \tens S)\tens S^{-1}\vect v=0,
\]
which is the same as
\[   
    \tens S \vect z'=\tens S\partial_t(\tens S^{-1}\vect v)
\]
since $\partial_t\vect v=\partial_t(\tens S\tens S^{-1}\vect v) = \tens S\partial_t(\tens S^{-1}\vect v) + (\partial_t\tens S)\tens S^{-1}\vect v$.
Multiplying by $\tens S^{-1}$ from the left and using that $z=\partial_t u$ we find
\[
	\partial_t(\vect u'-\tens S^{-1}\vect v)=0.
\]
By the initial condition $\vect u'|_{t=0}=\tens S^{-1}\vect v|_{t=0}$, we have $\vect u'=\tens S^{-1}\vect v$ for all $t$ which is equivalent to $\vect v=\tens S \vect u'$.
Hence, replacing $z$ by $\partial_t u$ as well as $\vect v$ by $\tens S  \vect u'$, the second equation of the system reads
\begin{eqnarray}
\nonumber
-\partial_t^2 u+2\sum_{i=1}^n g_i\partial_{x_i}\partial_t u +a\partial_t u+\sum_{i=1}^n b_i\partial_{x_i}u+cu \\ 
\nonumber+\sum_{i,j,k=1}^n\big( S_{ij}\partial_{x_i}(S_{jk}\partial_{x_k}u)+ (\partial_{x_i}S_{ij})S_{jk}\partial_{x_k}u- \partial_{x_i}(S_{ij}S_{jk})\partial_{x_k}u\big)
=f,
\end{eqnarray}
which is the same as
$$-\partial_t^2 u+2\sum_{i=1}^n g_i\partial_{x_i}\partial_t u+\sum_{i,j,k=1}^n S_{ik}S_{kj}\partial_{x_i}\partial_{x_j} u + a\partial_t u+\sum_{i=1}^n b_i\partial_{x_i} u +cu =f.
$$
Moreover, the condition $(u,\partial_t u)|_{t=0}=(u_0,u_1)$ is a direct consequence of the initial condition for the system. 
\end{proof}

\paragraph{Equivalence}

The content of Lemmas \ref{wave to system} and \ref{system to wave} can be summarized as follows. The problem of finding a solution to the Cauchy problem for the wave equation \eqref{wave_eq} is equivalent to the problem of finding a solution to the corresponding Cauchy problem for a first-order system \eqref{system}. Uniqueness of solutions is preserved during the rewriting process as well, more precisely we have the following statement. \medskip

\begin{theorem}\label{equi_thm}
	Given a wave equation \eqref{wave_eq} and the corresponding first-order system \eqref{system}. Let $u_0,u_1\in \mathcal G(\mathbb R^n)$. Then for functions $u\in\mathcal G(\mathbb R^{n+1})$ and $\vect w\in \mathcal G(\mathbb R^{n+1})^{n+2}$ such that $\vect w=(u,\partial_t u,\tens S\vect u')^T$ the following are equivalent:
	\begin{enumerate}[(i)]
		\item The function $u$ is the unique solution to the wave equation \eqref{wave_eq} with initial condition $(u,\partial_t u)|_{t=0}=(u_0,u_1)$.
		\item The function $\vect w$ is the unique solution to the first-order system \eqref{system} with initial condition $\vect w|_{t=0}=(u_0,u_1,\tens S \vect u_0')^T$.
	\end{enumerate}
\end{theorem}

\begin{proof}
	The translation of solutions between wave equations and first-order systems is an immediate consequence of Lemmas \ref{wave to system} and \ref{system to wave}. To show uniqueness take the following considerations into account:

	By Lemma \ref{wave to system}, two distinct solutions to the initial value problem \eqref{wave_eq} would give rise to two distinct solutions to \eqref{system}, since $u\mapsto\vect w=(u,\partial_t u,\tens S\vect u')^T$ is injective, thus contradicting unique solvability of the initial value problem for \eqref{system}.
	
	Suppose there were two distinct solutions $\vect w$ and $\widetilde{\vect w}$ to the initial value problem of \eqref{system} with $\vect w|_{t=0}=\widetilde{\vect w}|_{t=0}=(u_0,u_1,\tens S \vect u_0')^T$. Then the first component of $\vect w$ and $\widetilde{\vect w}$ would give two distinct solutions to \eqref{wave_eq}. But since the solution of \eqref{wave_eq} is unique, the first component of $\widetilde{\vect w}$ must be equal to the first component of $\vect w$. From the proof of Lemma \ref{system to wave} it is then clear that $\tilde z=z$ and $\tilde {\vect v}=\vect v$, hence $\widetilde {\vect w}=\vect w$.  
\end{proof}

Theorem \ref{equi_thm} guarantees that in the context of the differential algebra $\mathcal{G}$ the Cauchy problem for the second-order wave equation \eqref{wave_eq} is equivalent in fairly general circumstances to that for the corresponding first-order system (\ref{system}-\ref{coefficents of system}) provided only the natural, and merely algebraic, consistency of initial data holds. This is not true in spaces of distributions. When the  coefficients are of low regularity, the transformation process may fail at various places, e.\,g. we might end up with differential equations that do not make sense in any distribution space.\medskip
 
\begin{example}\label{Gegenbeispiel}
	Consider a wave equation used in linear acoustics: Let $p:{\mathbb R}^2\to{\mathbb R}$ and let $c,\rho:{\mathbb R}\to{[a,b]}$ with $a>0$. The acoustic wave equation reads
	\begin{equation}\label{acoustic equation}
		\partial_t^2p-c^2\rho\partial_x\left(\frac 1 {\rho}\partial_xp\right)=0.
	\end{equation}
	We define $\vect w=(p,\partial_t p,c\partial_x p)$ and formally obtain the symmetric hyperbolic system
	\begin{subequations}\label{acoustic system}
	\begin{align}
		-\partial_t w_1+w_2 = & 0, \label{acoustic system 1} \\
		-\partial_t w_2+c\partial_x w_3-(c'+c\cdot (\ln\rho)')w_3= & 0, \label{acoustic system 2} \\
		-\partial_t w_3+cw_2 = & 0. \label{acoustic system 3}
	\end{align}
\end{subequations}
	Equation \eqref{acoustic equation} can be regarded as an equality in $L^2({\mathbb R}^2)$, if we have $p\in H^2({\mathbb R}^2)$, \linebreak $\rho\in\mathrm{Lip}({\mathbb R}^2)$ and $c\in L^\infty({\mathbb R}^2)$. We then have $\partial_t p\in H^1({\mathbb R}^2)$ and $c\partial_x p \in L^2(\mathbb R^2)$, since $c\in L^\infty({\mathbb R}^2)$ and $\partial_x p\in H^1({\mathbb R}^2)$. Thus $\vect w\in H^2({\mathbb R}^2)\times H^1({\mathbb R}^2)\times L^2({\mathbb R}^2)$, but not better in general. Hence equation \eqref{acoustic system 2} will in general not be defined on the level of distributions. For example, if $c(x)=1+H(x)$, then $c'w_3$ would be a product of $\delta$ with an $L^2$-function.
\end{example}

\section{Existence and uniqueness for the Cauchy problems}\label{Existence}

Rewriting a wave equation  \eqref{wave_eq} as a first-order system \eqref{system} via Theorem \ref{equi_thm} allows to apply existence and uniqueness theorems for the latter to prove existence and uniqueness of a solution to the initial value problem for the wave equation. More precisely, let a vector $\vect w$ of generalized functions with representative $(\vect w_{\varepsilon})_\varepsilon=((u_{\varepsilon},z_{\varepsilon},\vect v_{\varepsilon})^T)_\varepsilon$ be given that is the unique solution  of the Cauchy problem for \eqref{system} with initial data $\boldsymbol{w}|_{t=0}=(u_0,u_1, \tens S \vect u_0')^T$, then Theorem \ref{equi_thm} implies that $(u_{\varepsilon})_\varepsilon$ will be the unique generalized solution to the Cauchy problem for \eqref{wave_eq} with initial data $(u,\partial_t u)|_{t=0}=(u_0,u_1)$. In the following theorem we will give conditions on the coefficients of a wave equation \eqref{wave_eq} that guarantee the existence of a unique generalized solution to the corresponding first-order problem and, hence, existence and uniqueness of a generalized solution to the wave equation \eqref{wave_eq}.

To this end, we are going to invoke the existence theory for symmetric hyperbolic systems developed in \cite{Obe88,Obe89,CoOb90,LaOb91,Hor04a} and, in particular, the existence results of \cite{HoSp12}, which we will now briefly summarize. We start by recalling the essential asymptotic conditions. Let $U_T:=(0,T)\times \mathbb R^n$. For a function $g$ on $U_T$, we introduce its mixed $L^1$-$L^\infty$-norm by $\|g\|_{L^{1,\infty}(U_T)}:=\int_0^T \|g (s,\cdot)\|_{L^{\infty}(\mathbb R^n)}\,\rm d s.$
A generalized function $f\in\mathcal G(U)$ is said to be of \emph{local $L^\infty$-log-type} (\cite[Definition 1.1]{Obe88}) if it admits a representative $(f_\varepsilon)_\varepsilon$ such that for all $K$ compact in $U$, we have $\|f_\varepsilon\|_{L^\infty(K)}=O(\log\frac 1 \varepsilon)$ as $\varepsilon\to 0$; $f\in\mathcal G(U)$ is said to be of \emph{$L^\infty$-log-type} (\cite[Definition 1.5.1]{GKOS01}), if it admits a representative $(f_\varepsilon)_\varepsilon$ such that $\|f_\varepsilon\|_{L^\infty(U)}=O(\log\frac 1 \varepsilon)$ as $\varepsilon\to 0$; $f\in\mathcal G_{L^\infty}(U_T))$ is said to be of \emph{$L^{1,\infty}$-log-type} (cf. \cite[Definition 2.1]{CoOb90}) if it admits a representative $(f_\varepsilon)_\varepsilon$ such that $\|f_\varepsilon\|_{L^{1,\infty}}=O(\log\frac 1 \varepsilon)$ as $\varepsilon\to 0$ .
 
Solution candidates to the Cauchy problem for symmetric hyperbolic systems with \linebreak Colombeau generalized coefficients \eqref{system} are obtained as a net of solutions to the family of  classical equations $-\partial_t\vect w_\varepsilon + \sum_{i=1}^n \boldsymbol{\mathsf A}_{i,\varepsilon}\partial_{x_i}\vect w_\varepsilon+\boldsymbol{\mathsf B_\varepsilon}\vect w_\varepsilon=\boldsymbol F_\varepsilon$. By imposing additional asymptotic growth conditions in $\varepsilon$ on the coefficient matrices, a Gronwall-type argument can be used to prove the moderateness of the family of smooth solutions, hence existence of generalized solutions. Uniqueness of generalized solutions amounts to stability of the family of smooth solutions under negligible perturbations of the data. For convenience of the reader, we combine results from \cite{HoSp12}, adjusted to the situation at hand, in the following theorem (cf. \cite[Theorems 3.1, 3.2, and 3.4]{HoSp12}).\medskip

\begin{theorem}\label{Th:hyperbolic system}
	Let $\tens A_i$, $\tens B\in M_{n+2}(\mathcal G_{L^\infty}(U_T))$, where $A_i$ is symmetric. Then we have the following three results.
	\begin{enumerate}[A)]
		\item The Cauchy problem for the system \eqref{system} with initial data $\vect w_0\in(\mathcal G({\mathbb R}^n))^{n+2}$ and right-hand side $\vect F\in(\mathcal G(\overline U_T))^{n+2}$ has a unique solution $\vect w\in(\mathcal G(U_T))^{n+2}$ if
		\begin{enumerate}[(i)]
			\item the spatial derivatives $\tens A'_i$ as well as $\frac 1 2(\tens B+\tens B^T)$, the symmetric part \linebreak of the matrix $\tens B$, are of local $L^\infty$-log-type,
			\item there exists some constant $R_{\tens A}>0$ such that we have $\sup_{(t,x)}|\tens A_{i,\varepsilon}(t,x)|=O(1)$ on $(0,T)\times\{x\in{\mathbb R}^n:|x|>R_{\tens A}\}$ as $\varepsilon\to 0$.
		\end{enumerate}
		\vspace{0.15cm}
		\item The Cauchy problem for the system \eqref{system} with initial data $\vect w_0\in(\mathcal G_{L^2}({\mathbb R}^n))^{n+2}$ and right-hand side $\vect F\in(\mathcal G_{L^2}(U_T))^{n+2}$ has a unique solution $\vect w\in(\mathcal G_{L^2}(U_T))^{n+2}$ if the spatial derivatives of $\tens A'_i$ as well as the symmetric part of the matrix $\tens B$ are of $L^{1,\infty}$-log-type.
		\vspace{0.15cm}
		\item Let initial data $\vect w_0\in(\mathcal G_{L^\infty}({\mathbb R}^n))^{n+2}$ and  right-hand side $\vect F\in(\mathcal G_{L^\infty}(U_T))^{n+2}$ be given. If the spatial derivatives of $\tens A'_i$ as well as the symmetric part of the matrix $\tens B$ are of $L^\infty$-log-type, then  there exists a unique solution $\vect w\in(\mathcal G(U_T))^{n+2}$ of \eqref{system} such that $\vect w|_{t=0}-\vect w_0\in(\mathcal N({\mathbb R}^n))^{n+2}$.  
		\end{enumerate}

\end{theorem}
\medskip
\begin{remark} 
	When considering case C, the situation occurs that the initial data $\vect w_0$ is an element of the algebra $(\mathcal G_{L^\infty}({\mathbb R}^n))^{n+2}$, whereas the restriction of the solution $\vect w$ to the initial surface, i.\,e. $\vect w|_{t=0}$ is in $(\mathcal G({\mathbb R}^n))^{n+2}$. This issue can be resolved in the following way: Every representative $(f_\varepsilon)_\varepsilon$ of a generalized function $f\in\mathcal G_{L^\infty}({\mathbb R}^n)$ is also moderate in the sense of $\mathcal G({\mathbb R}^n)$; thus, $\mathcal G_{L^\infty}({\mathbb R}^n)$ can be interpreted as a subset of $\mathcal G({\mathbb R}^n)$ if we allow the difference of two representatives of $f$ to be in the ideal $\mathcal N({\mathbb R}^n)$ instead of $\mathcal N_{L^\infty}({\mathbb R}^n)$. So, we obviously have that $\vect w|_{t=0}-\vect w_0\in(\mathcal N({\mathbb R}^n))^{n+2}$ but not necessarily in $(\mathcal N_{L^\infty}({\mathbb R}^n))^{n+2}$. In other words, we consider the initial data to be in the algebra $(\mathcal G({\mathbb R}^n))^{n+2}$ but additionally satisfying the moderateness estimates of $(\mathcal G_{L^\infty}({\mathbb R}^n))^{n+2}$.
\end{remark}
Finally, we are able to formulate an existence and uniqueness theorem for wave equations based on Theorems \ref{equi_thm} and \ref{Th:hyperbolic system}.\medskip

\begin{theorem}\label{Th:Wave from system}
Consider the Cauchy problem
\begin{equation}\label{wave equation in theorem}
	-\partial_t^2 u+2\sum_{i=1}^n g_i\partial_{x_i}\partial_t u+\sum_{i,j=1}^n R_{ij}\partial_{x_i}\partial_{x_j} u + a\partial_t u+\sum_{i=1}^n b_i\partial_{x_i} u +cu =f
\end{equation}
and		
\[
	(u,\partial_tu)|_{t=0}= (u_0,u_1)
\]
with coefficients $R_{ij}, g_i, a,b_i,c$ in $\mathcal G_{L^{\infty}}(U_T)$ and $\tens R$ positive definite. Let, furthermore, $\tens S=\tens R^{\frac 12}$, where we take the square root via diagonalization of $\tens R$. Then we have the following three results. 
\begin{enumerate}[A)] 
	\item The Cauchy problem \eqref{wave equation in theorem} with initial data $u_0,u_1\in\mathcal G(\mathbb R^n)$ and right-hand side \linebreak $f\in\mathcal G(U_T) $ has a unique solution $u\in\mathcal G(U_T)$ if
	 \begin{enumerate}[(i)]
		\item the lower order coefficients $a$, $c$, $\vect b$, as well as $\tens S$, the derivative $\rm d\tens S$, the inverse $\tens S^{-1}$ and $\tens g'$ are  of local $L^\infty$-log-type,
		\item there exists $R_{\tens S,\vect g}>0$ such that we have $\sup_{(t,x)}|\vect g_\varepsilon(t,x)|=O(1)$ and \linebreak $\sup_{(t,x)}|\tens S_\varepsilon(t,x)|=O(1)$ on $(0,T)\,\times\{x\in\mathbb R^n| |x|>R_{\tens S,\vect g}\}$ as $\varepsilon\to 0$.
	\end{enumerate}
	\vspace{0.15cm}
	\item The Cauchy problem \eqref{wave equation in theorem} with initial data $u_0,u_1\in\mathcal G_{L^2}(\mathbb R^n)$ and right-hand side $f\in\mathcal G_{L^2}(U_T) $ has a unique solution $u\in\mathcal G_{L^2}(U_T)$ if the lower order coefficients $a$, $c$, $\vect b$, as well as $\tens S$, the derivative $\rm d\tens S$, the inverse $\tens S^{-1}$ and $\tens g'$ are of $L^{1,\infty}$-log-type.
		\vspace{0.15cm}
	\item Let initial data $u_0,u_1\in\mathcal G_{L^\infty}(\mathbb R^n)$ and right-hand side $f\in\mathcal G_{L^\infty}(U_T)$ be given. If the lower order coefficients $a$, $c$, $\vect b$, as well as $\tens S$, the derivative $\rm d\tens S$, the inverse $\tens S^{-1}$ and $\tens g'$ are of $L^\infty$-log-type, then there exists a unique solution $u\in\mathcal G(U_T)$ of the wave equation \eqref{wave equation in theorem} such that $(u,\partial_t u)|_{t=0}-(u_0,u_1)\in(\mathcal N({\mathbb R}^n))^{2}$.
\end{enumerate}
\end{theorem}
\begin{proof}
	We start with the proof for case A. We rewrite the wave equation into the corresponding symmetric hyperbolic system with $\tens A_i$, $\tens B$ and $\vect F$ as in \eqref{coefficents of system}.  Clearly the coefficients of the hyperbolic system are in $\mathcal G_{L^\infty}({\mathbb R})$ since the coefficients of the wave equation are. From condition (i) and the structure of \eqref{coefficents of system} we obtain that $\tens A'_i$ and $\frac 1 2(\tens B+\tens B^T)$---the symmetric part of $\tens B$---are locally of $L^\infty$-log-type. Since by condition (ii) $u_0$ and $u_1$ are in $\mathcal G({\mathbb R}^n)$, and $\tens S$ has entries in $\mathcal G_{L^\infty}(U_T)$, the initial data for the system $\vect w_0=(u_0,u_1,\tens Su'_0)^T$ is in $(\mathcal G({\mathbb R}^n))^{n+2}$. Furthermore, $f\in\mathcal G(U_T)$, thus $F=(0,f,0)\in (\mathcal G(U_T))^{n+2}$. The matrix $\tens S$ and the vector $\vect g$ satisfy condition (ii). Thus, there exists a constant $R_{\tens S,\vect g}>0$ such that $\tens A$, which depends only on $\tens S$ and $\vect g$, is $O(1)$ on $(0,T)\,\times\{x\in\mathbb R^n| |x|>R_{\tens S,\vect g}\}$. Summing up, all conditions of Theorem \ref{Th:hyperbolic system}, case A are satisfied, and we can apply the theorem to obtain a solution $\vect w$ to the initial value problem of the hyperbolic system.  Theorem \ref{equi_thm} guarantees that the first component $u$ of $\vect w$ is the unique solution to the Cauchy problem for the wave equation.
	
	 The proofs for cases B [resp. C] follow the same pattern. Again, we rewrite the wave equation into its corresponding symmetric hyperbolic first-order system and obtain matrices $\tens A_i$ and $\tens B$ in $M_{n+2}(\mathcal G_{L^\infty}(U_T))$. By condition (i) we have that $\tens A'_i$ and the symmetric part of $\tens B$ are $L^{1,\infty}$-log-type [resp. $L^\infty$-log-type]. Since by condition (ii) the initial data $u_0$ and $u_1$ are in $\mathcal G_{L^2}({\mathbb R}^n)$ [resp. $\mathcal G_{L^\infty}({\mathbb R}^n)$], and $\tens S$ has entries in $\mathcal G_{L^\infty}(U_T)$, the initial data for the system $\vect w_0=(u_0,u_1,\tens Su'_0)^T$ is in $(\mathcal G_{L^2}({\mathbb R}^n))^{n+2}$  [resp. $(\mathcal G_{L^\infty}({\mathbb R}^n))^{n+2}$]. We also have  $f\in\mathcal G_{L^2}(U_T)$ [resp.  $\mathcal G_{L^\infty}(U_T)$], thus $\vect F=(0,f,0)\in(\mathcal G_{L^2}(U_T))^{n+2} $ [resp. $ (\mathcal G_{L^\infty}(U_T))^{n+2}$]. Altogether we can apply Theorem \ref{Th:hyperbolic system}, case B [resp. C] and obtain a solution $\vect w$ to the initial value problem for the hyperbolic system. Finally, Theorem \ref{equi_thm} guarantees its first component $u$ is the unique solution to the Cauchy problem for the wave equation, and we are done.
\end{proof}

The asymptotic estimates on the coefficients required in Theorem \ref{Th:Wave from system} are less restrictive than those supposed in the (local) existence results for the initial value problem for the wave equation on ``weakly singular'' Lorentzian manifolds, cf.\ \cite[Theorem 3.1]{GMS09} and
\cite[Theorem 3.1]{Han11}. In particular, the conditions of Theorem \ref{Th:Wave from system} A) are general enough to cover the Laplace-Beltrami operator of metrics in the Geroch-Traschen class, which is was not possible previously. The relevance of this class, introduced in \cite{GeTr87}, comes from the fact that it is viewed as the ``maximal reasonable'' class of metrics that allows for the definition of the Riemann curvature tensor as a distribution (see also \cite{LeMa07}). We finish this paper by deriving an existence result for the wave operator of such metrics.

A Lorentzian metric $\tens{g}_0$ on a smooth manifold $M$ belongs to the Geroch-Traschen class if $\tens{g}_0$ and its inverse $\tens{g}_0^{-1}$ belong to $\GT(M)$ and, furthermore, if $|\det\tens{g}_0|\geq C>0$ almost everywhere on compact sets. Since we are only interested in a local existence result we may work in a fixed chart and moreover cut off the metric $\tens{g}_0$ outside some ball such that it is constant there. We then regularize $\tens{g}_0$ via componentwise convolution with a mollifier to obtain a generalized metric. More precisely, denoting the components of $\tens{g}_0$ by $g_{0,ij}$ we will write $g^\varepsilon_{ij}$ for their smoothings, i.\,e. $g^\varepsilon_{ij}=g_{0,ij}*\psi_{\varepsilon}$, with $(\psi_\varepsilon)_\varepsilon$ being a model delta net (i.\,e. $\psi_\varepsilon(x)=\varepsilon^{-(n+1)}\rho(x/\varepsilon)$ for some fixed test function $\rho$ with unit integral). For a more sophisticated way of smoothing metrics of the Geroch-Traschen class see \cite{StVi09}. We denote by $\tens{g}=[(\tens{g}_\varepsilon)_\varepsilon]$ the resulting generalized Lorentzian metric on $\mathbb{R}^{n+1}$. It is then clear that in general $\mathrm d\tens{g}_\varepsilon\not=O(1)$ (otherwise $\tens{g}\in\Wl^{1,\infty}$), hence condition (B) of \cite[Theorem 3.1]{GMS09} is violated as well as condition (i) of \cite[Theorem 3.1]{Han11}.  However, logarithmically rescaling the mollifier (i.\,e. setting $\psi_\varepsilon(x)=\gamma_\varepsilon^{n+1}\rho(\gamma_\varepsilon x)$ with $\gamma_\varepsilon=\log\frac{1}{\varepsilon}$) allows to apply Theorem \ref{Th:Wave from system} A).\medskip 

\begin{corollary}[The wave equation for Geroch-Traschen metrics]
 Let $\tens{g}$ be a generalized metric on $\mathbb{R}^{n+1}$ obtained as the smoothing of a metric of Geroch-Traschen class as above with a logarithmically rescaled mollifier. Then the Cauchy problem for the wave equation
  \[
   \Box_{\tens{g}} u=0\qquad (u,\partial_t u)\mid_{t=0}=(u_0,u_1) \in\mathcal G(\mathbb R^n)
  \]
  has a unique solution in $u\in\mathcal G(U_T)$ for arbitrary $T$.
\end{corollary}

\medskip\noindent
\emph{Sketch of proof. } We check that the conditions of Theorem \ref{Th:Wave from system} A) hold. Indeed $a, \vect{b}$ and $c$ vanish as well as $f$. Furthermore $R_{ij}$ and $g_i$ are given by the components of $\tens{g}$ which belong to $\mathcal G_{L^{\infty}}(U_T)$ due to the cut off applied to the metric $\tens{g}_0$. Condition A)(ii) holds true again due to the cut off and local boundedness of $\tens{g}_0$. As for condition A)(i), $\tens{S}$ and its inverse $\tens{S}^{-1}$ are even locally uniformly bounded. Finally, due to the logarithmic rescaling of the mollifier the derivatives of $\tens{g}$    
are of local $L^\infty$-$\log$-type and we are done.\hfill$\Box$

\section*{Acknowledgements}
This work was supported by FWF-Grants Y237, P20525 and P23714 as well as the Research Grant 2011 of the University of Vienna. We furthermore would like to thank Michael Grosser and Michael Kunzinger for their valuable input.

\end{document}